\documentclass[11pt]{amsart}
\usepackage[utf8]{inputenc}
\usepackage{fourier}
\usepackage{url}
\usepackage{amssymb}
\usepackage{enumerate}
\usepackage[margin=3cm]{geometry}
\usepackage{graphicx} 
\usepackage{todonotes}
\usepackage{comment}
\usepackage{placeins}
\usepackage{sidecap}
\sidecaptionvpos{figure}{t}

\newtheorem{theorem}{Theorem}
\newtheorem{proposition}[theorem]{Proposition}
\newtheorem{lemma}[theorem]{Lemma}
\newtheorem{corollary}[theorem]{Corollary}

\theoremstyle{remark}
\newtheorem{example}[theorem]{Example}
\newtheorem{remark}[theorem]{Remark}

\usepackage[pdf]{graphviz}

\makeatletter
\def\@bignumber#1#2{%
  \ifx#2\end
    #1\let\next\@gobble
  \else
    #1\hspace{0pt plus 1pt}\let\next\@bignumber
  \fi
  \next#2}
\newcommand{\bignumber}[1]{\@bignumber#1\end}
\makeatother

\title{The isomorphism problem for ideal class monoids of numerical semigroups}
\author{P. A. García-Sánchez}
\address{Departamento de \'Algebra and IMAG, Universidad de Granada, E-18071 Granada, Espa\~na}
\email{pedro@ugr.es}

\date{today}



\begin{document}

\begin{abstract}
    From any poset isomorphic to the poset of gaps of a numerical semigroup $S$ with the order induced by $S$, one can recover $S$. As an application, we prove that two different numerical semigroups cannot have isomorphic posets (with respect to set inclusion) of ideals whose minimum is zero. We also show that given two numerical semigroups $S$ and $T$, if their ideal class monoids are isomorphic, then $S$ must be equal to $T$.
\end{abstract}

\maketitle

\section{Introduction}

A \emph{numerical semigroup} $S$ is a submonoid of $(\mathbb{N},+)$ such that $\mathbb{N}\setminus S$ has finitely many elements, where $\mathbb{N}$ denotes the set of non-negative integers. A set of integers $I$ is said to be an \emph{ideal} of $S$ if $I+S\subseteq I$ and $I$ has a minimum (see for instance \cite[Chapter~3]{ns-ap} for some basic background on ideals of numerical semigroups). On the set of ideals of $S$, we define the following relation: $I\sim J$ if there exists an integer $z$ such that $I=z+J$. The set of ideals modulo this equivalence relation is known as the \emph{ideal class monoid} of $S$, denoted $\mathcal{C}\ell(S)$. Addition of two classes $[I]$ and $[J]$ is defined in the natural way: $[I]+[J]=[I+J]$. 

The ideal class monoid of a numerical semigroup was introduced in \cite{b-k} inspired by the definition of ideal class group of a Dedekind domain. In \cite{c-da-gs}, we proved that from some combinatorial properties of the ideal class monoid of a numerical semigroup we can recover relevant information of the numerical semigroup like, for instance, its genus, multiplicity, type, and number of unitary extensions.

We say that an ideal $I$ of $S$ is \emph{normalized} if $\min(I)=0$; we denote by $\mathfrak{I}_0(S)$ the set of normalized ideals $I$ of $S$. The map $\mathcal{C}\ell(S)\to \mathfrak{I}_0(S)$, $[I]\mapsto -\min(I)+I$ is bijective. Moreover, for $I,J\in \mathfrak{I}_0(S)$, the ideal $I+J$ is also in $\mathfrak{I}_0(S)$. Thus, the mapping $[I]\mapsto -\min(I)+I$ is a monoid isomorphism. 

In \cite[Section~4.3]{c-da-gs}, we studied some of the properties of the poset $(\mathfrak{I}_0(S),\subseteq)$. It is natural to wonder if $(\mathfrak{I}_0(S),\subseteq)$ completely determines $S$ in the following sense: if $T$ is a numerical semigroup and $(\mathfrak{I}_0(S),\subseteq)$ is isomorphic to $(\mathfrak{I}_0(T),\subseteq)$, then $S=T$? Recall that two posets $(P,\le_P)$ and $(Q,\le_Q)$ are \emph{isomorphic} if there exists an order isomorphism $f$ from $P$ to $Q$, that is, $f$ is bijective and for every $a,b\in P$, $a\le_P b$ if and only if $f(a)\le_Q f(b)$. We translate this problem of poset isomorphism of normalized ideals of a numerical semigroup to an isomorphism problem of posets of gaps with respect to the order induced by the semigroup.

For a numerical semigroup $S$, the order induced by $S$ on the set of integers, denoted $\le_S$, is defined as $a\le_S b$ if $b-a\in S$. The poset $(\mathbb{Z},\le_S)$ (with $\mathbb{Z}$ the set of integers) has been studied for several families of numerical semigroups, and more particularly the M\"obius function associated to $\le_S$ (see \cite{c-ra} or \cite{c-gm-ra} for a generalization to affine semigroups).

The set $\operatorname{G}(S)=\mathbb{N}\setminus S$ is the \emph{gap set} of $S$; its elements are called \emph{gaps} of $S$. It was already shown in \cite[Proposition~2.6]{b-k} that the set $\mathcal{C}\ell(S)$ is in one-to-one correspondence with the set of antichains of gaps with respect to $\le_S$ (these antichains are called $S$-Leans in \cite{mm-u}). For every gap $g$ of $S$, the set $\{0,g\}+S$ is an ideal of $S$, and so $\operatorname{G}(S)$ is embedded naturally in $\mathfrak{I}_0(S)$. Moreover, if $g'$ is another gap of $S$, then $g\le_S g'$ if and only if $\{0,g'\}+S\subseteq \{0,g\}+S$ (Lemma~\ref{lem:prec-inc-principal}). 
If we are able to characterize the ideals of the form $\{0,g\}+S$ from their properties in the poset $(\mathfrak{I}_0(S),\subseteq)$, then we can extract a poset isomorphic to $(\operatorname{G}(S),\le_S)$ and thus recover $S$. This is actually the strategy we use to prove that if the posets $(\mathfrak{I}_0(S),\subseteq)$ and $(\mathfrak{I}_0(T),\subseteq)$ are isomorphic, then $S$ and $T$ must be equal.

The ideal class monoid of a numerical semigroup is a monoid. Thus, it is natural to ask if two different numerical semigroups will have isomorphic ideal class monoids \cite[Question~6.1]{c-da-gs}. The answer is no. Theorem~\ref{thm:isomorphic-class-monoids-addition} states that if $S$ and $T$ are numerical semigroups, and their ideal class monoids are isomorphic, then $S$ and $T$ must be equal.

It order to solve the isomorphism problem for ideal class monoids of numerical semigroups, we study what are the consequences of having an isomorphism between $(\mathfrak{I}_0(S),+)$ and $(\mathfrak{I}_0(T),+)$, with $S$ and $T$ numerical semigroups. In particular, we show that over-semigroups of $S$ are in correspondence with over-semigroups of $T$, and their corresponding ideal class monoids must be isomorphic. 

Most of the computations in the examples presented in this manuscript where performed using the \texttt{GAP} \cite{GAP4} package \texttt{numericalsgps} \cite{numericalsgps}. The code used for these calculations can be found at 

\centerline{\url{https://github.com/numerical-semigroups/ideal-class-monoid}.} 

The package \texttt{numericalsgps} was also used to draw Hasse diagrams of the posets mentioned above for several numerical semigroups, providing in this way clues on what where the results needed to proof our main theorems.






\section{Determining a numerical semigroup from the order induced in its gap set}

Suppose that we are given a numerical semigroup $S$ as a sequence $\{s_0,s_1,\dots,s_n,\dots\}$ of which the only data we know is whether $s_i\le_S s_j$ for $i,j\in \mathbb{N}$. The $\nu$ sequence $\nu_i= |\{j \in \mathbb{N}: s_j\le_S s_i\}|$ completely determines $S$ (see \cite{ba}; here $|X|$ denotes the cardinality of the set $X$). Thus, if $S$ and $T$ are numerical semigroups whose respective posets $(S,\le_S)$ and $(T,\le_T)$ are isomorphic, then $S=T$.

Now, suppose that what we have is an enumeration $H=\{h_1,\dots,h_g\}$ of the gap set of $S$, $\operatorname{G}(S)$, and how these elements are arranged with respect to $\le_S$. We want to recover $S$ from this information. 

Recall that the \emph{multiplicity} of $S$ is the least positive integer in $S$. From \cite[Lemma~2.5(1)]{b-k}, we know that multiplicity of $S$ is the cardinality of $\operatorname{Minimals}_{\le_S}(H)$ plus one. The argument used in the proof of that lemma also shows that the maximal number of elements in an antichain (with respect to $\le_S$) is precisely the multiplicity of the semigroup minus one.


For $h\in H$, and inspired by the $\nu$ sequence described above, define 
\[
\operatorname{D}_H(h)=\{ h'\in H : h'\le_S h\},
\]
and set $\operatorname{nd}_H(h)=|\operatorname{D}_H(h)|$. In particular,
\begin{equation}\label{eq:minimals-nd-1}
    \operatorname{Minimals}_{\le_S}(H)=\{ h\in H : \operatorname{nd}_H(h)=1\}.
\end{equation}

As a consequence of the following result, the map $\operatorname{nd}_H$ is non-decreasing.

\begin{lemma}\label{lem:divisors-gap}
  Let $h\in H$. Then, 
  \[\operatorname{D}_H(h)=\{ h-s : s\in S\cap[0,h] \}.\]
  In particular, $\operatorname{nd}_H(h)=|S\cap[0,h]|$, and if $h'\in H$, with $h<h'$, then 
  \[
    |S\cap [h,h']| = \operatorname{nd}_H(h')-\operatorname{nd}_H(h).
\]
\end{lemma}
\begin{proof}
    Take $h'\in H$. Then, $h'\le_S h$ if and only if $h-h'=s$ for some $s\in S$. Clearly, $s\in S\cap [0,h]$ and $h'=h-s$. 
    
    Now, take $t\in S\cap[0,h]$. Then, $h-t\not\in S$, since otherwise $h$ would be in $S$. So $h-t\in \operatorname{D}_H(h)$. 
    
    The second assertion follows from the first.
\end{proof}

Let $S$ be a numerical semigroup. A (finite) \emph{run of elements} in $S$ is an interval $\{s,s+1,\ldots,s+k\}$ of elements of $S$ such that $s-1\not\in S$ and $s+k+1\not\in S$. Analogously, a \emph{run of gaps} of $S$, or \emph{desert}, is an interval $\{h,h+1,\ldots,h+l\}$ of gaps of $S$ such that $h-1\in S$ and $k+l+1\in S$. Let $\operatorname{C}(S)$ be the \emph{conductor} of $S$, that is, the least integer $c$ such that $c+\mathbb{N}\subseteq S$. The numerical semigroup $S$ can be expressed as $S=S_0\cup S_1\cup\dots \cup S_r \cup \operatorname{C}(S)+\mathbb{N}$, such that $\max(S_i)+1<\min(S_{i+1})$  for all $i$, that is, all the elements in $S_i$ are smaller than those in $S_{i+1}$ and there is at least a gap of $S$ between these two sets. If $S\neq\mathbb{N}$, then $S_0=\{0\}$.
 

\begin{theorem}\label{th:description-runs-elements}
Let $S$ be a numerical semigroup, $S\neq \mathbb{N}$.
\begin{enumerate}[(1)]
    \item If $R$ is a run of gaps and $h\in R$, then 
    \[
        R=\{h'\in H : \operatorname{nd}_H(h)=\operatorname{nd}_H(h')\}.
    \]

    \item If $R$ is a run of elements of $S$, with $R\neq \{0\}$, then
    \[
        R=\{\operatorname{nd}_H(h)+d,\operatorname{nd}_H(h)+d+1,\dots,\operatorname{nd}_H(h')+d-1\},
    \]
    with $h=\min(R)-1$, $h'=\max(R)+1$ and $d=|\{g \in \operatorname{G}(S) : \operatorname{nd}_H(g)\le \operatorname{nd}_H(h)\}|$.
\end{enumerate}
\end{theorem}
\begin{proof}
    Set $H=\mathbb{N}\setminus S$, and let $R$ be a run of gaps. By Lemma~\ref{lem:divisors-gap}, the map $\operatorname{nd}_H$ is non-decreasing and it is constant when restricted to a desert. Moreover, two gaps $h$ and $h'$ are in the same desert if and only if $\operatorname{nd}_H(h)=\operatorname{nd}_H(h')$. Therefore, the first assertion follows.
    
    Now, let $R$ be a run of elements of $S$ with $h$ and $h'$ as in the hypothesis of the second assertion. Let $r=\min(R)$, and so $h=r-1$. From the previous paragraph, we deduce that $d=|H\cap [0,h]|$. Let $t=|S\cap [0,h]|$. Clearly, $h+1=d+t$, and thus $r=d+t$. 
    By Lemma~\ref{lem:divisors-gap}, $t=|S\cap [0,h]|=\operatorname{nd}_H(h)$. 
    Finally, by using again Lemma~\ref{lem:divisors-gap}, we have that $|R|=|S\cap [h,h']|=\operatorname{nd}_H(h')-\operatorname{nd}_H(h)$, and this completes the proof by taking into account that $R$ is an interval of elements in $S$.
\end{proof}

\begin{remark}\label{rem:construct-from-gaps}
    Notice that Theorem~\ref{th:description-runs-elements} is telling us that if we know $\operatorname{nd}_H:H\to \mathbb{N}$, then we can fully reconstruct $S$.
\end{remark}

\begin{example}
Assume that $H=\{g_1,\dots,g_8\}$ with $\operatorname{nd}_H(g_1)=\dots=\operatorname{nd}_H(g_4)=1$;  $\operatorname{nd}_H(g_5)=2$; $\operatorname{nd}_H(g_6)=\operatorname{nd}_H(g_7)=3$; and  $\operatorname{nd}_H(g_8)=6$. 

By \cite[Lemma~2.5(1)]{b-k} and \eqref{eq:minimals-nd-1}, from $\operatorname{nd}_H(g_1)= \dots = \operatorname{nd}_H(g_4)=1$, we know that the multiplicity of $S$ is five. In light of Theorem~\ref{th:description-runs-elements}, $S_1=\{5\}=\{1+4,\dots,2+4-1\}$, $S_2=\{7\}=\{2+5,\dots,3+5-1\}$, $S_3=\{10,11,12\}=\{3+7,\ldots,6+7-1\}$, and the last desert is $\{13\}$, since there is only one gap having $\operatorname{nd}_H(h)=6$. Thus, $S=\{0,5,7,10,11,12\}\cup (14+\mathbb{N})$.


\end{example}

\begin{corollary}\label{cor:gap-poset-isomorphic}
    Let $S$ and $T$ be numerical semigroups. If the posets $(\operatorname{G}(S),\le_S)$ and $(\operatorname{G}(T),\le_T)$ are isomorphic, then $S=T$.
\end{corollary}

Let $\operatorname{PF}(S)=\operatorname{Maximals}_{\le_S}(\mathbb{Z}\setminus S)$, which is known as the set of \emph{pseudo-Frobenius numbers} of $S$. The cardinality of $\operatorname{PF}(S)$ is the \emph{type} of $S$, denoted $\operatorname{t}(S)$. 
The \emph{Frobenius number} of $S$, defined as $\operatorname{F}(S)=\max(\mathbb{Z}\setminus S)$, is always a pseudo-Frobenius number, and so the type of a numerical semigroup is a positive integer. Clearly, $\operatorname{C}(S)=\operatorname{F}(S)+1$.

Notice that if we consider the Hasse diagram of $(\operatorname{G}(S),\le_S)$ as an undirected graph, then this graph has at most $\operatorname{t}(S)$ connected components.

\begin{example}
    Let $H=\{1, 2, 3, 4, 5, 9, 10 \}$, which is the set of gaps of $S=\{ 0, 6, 7, 8\}\cup (11+\mathbb{N})$. Then the Hasse diagram of $H$ with respect to $\le_S$ looks like this:
\begin{center}
\includegraphics[scale=0.5]{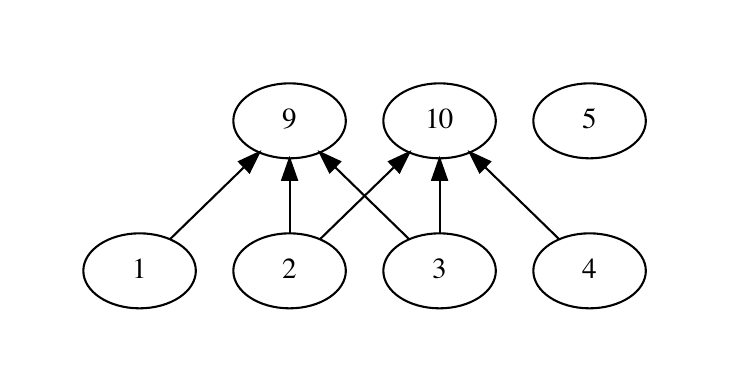}
\end{center}

The type of $S$ is three and the undirected graph has two connected components.
\end{example}


Recall that an affine semigroup is a finitely generated submonoid of $(\mathbb{N}^n,+)$ for some positive integer $n$. The poset of the set of gaps does not uniquely determine an affine semigroup as the following example shows. 

\begin{example}
The affine semigroups $\mathbb{N}^2\setminus\{(1,0),(0,1)\}$ and $\mathbb{N}^2\setminus\{(0,1),(0,2)\}$ have isomorphic posets of gaps, but the first is minimally generated by seven elements, while the second is by six, and thus they cannot be isomorphic. 
\end{example}

\section{The poset of normalized ideals of a numerical semigroup under inclusion}

Let $S$ be a numerical semigroup. 
Recall that the set of normalized ideals of $S$ is 
\[
\mathfrak{I}_0(S)=\{ I \subseteq \mathbb{N} : I + S \subseteq I, \min(I)=0\}.
\]

For $I\in \mathfrak{I}_0(S)$, set $I^*=I\setminus\{0\}$; in particular, $S^*=S\setminus\{0\}$.



\begin{lemma}\label{lem:prec-inc-principal}
    Let $S$ be a numerical semigroup and let $g$ and $g'$ be gaps of $S$. Then, $g\le_S g'$ if and only if $\{0,g'\}+S\subseteq \{0,g\}+S$.
\end{lemma}
\begin{proof}
    Notice that $\{0,g'\}+S\subseteq \{0,g\}+S$ if and only if $g'\in \{0,g\}+S$, or equivalently, $g'=g+s$ for some $s\in S$, and this means that $g\le_S g'$.
\end{proof}

Define
\[\mathfrak{P}_0(S)=\left\{\{0,g\}+S : g\in H\right\}.\]
In light of Lemma~\ref{lem:prec-inc-principal}, the poset $(\mathfrak{P}_0(S),\supseteq)$ is isomorphic to $(H,\le_S)$.
Thus, if we find a way to recover the set $\mathfrak{P}_0(S)$ from $(\mathfrak{I}_0(S),\subseteq)$, we will be able to recover $S$ from $(\mathfrak{I}_0(S),\subseteq)$ by using Remark~\ref{rem:construct-from-gaps}. 

Observe that for $I\in\mathfrak{I}_0(S)$, we have that $I=\operatorname{Minimals}_{\le_S}(I)+S$, and that every $X\subseteq N$, for which $I=X+S$ holds, contains $\operatorname{Minimals}_{\le_S}(I)$. Thus, $\operatorname{Minimals}_{\le_S}(I)$ is a minimal generating system of $I$ and it is included in $\operatorname{G}(S)\cup\{0\}$. Also, $\operatorname{Minimals}_{\le_S}(I)=I\setminus (I+S^*)$. The elements of $\operatorname{Minimals}_{\le_S}(I)$ are called the \emph{minimal generators} of $I$.

\begin{lemma}\label{lem:remove-min-gen}
    Let $S$ be a numerical semigroup and let $I\in \mathfrak{I}_0(S)$. For $x\in I^*$, the set $I\setminus\{x\}\in \mathfrak{I}_0(S)$ if and only if $x$ is a minimal generator of $I$. 
\end{lemma}
\begin{proof}
    If $x\not\in\operatorname{Minimals}_{\le_S}(I)$, then $x=g+s$, with $g\in\operatorname{Minimals}_{\le_S}(I)$ and $s\in S^*$. Hence, $x=g+s\not\in I\setminus\{x\}$, and consequently $I\setminus\{x\}$ is not an ideal of $S$.

    If $I\setminus\{x\}$, with $x \in \operatorname{Minimals}_{\le_S}(I)$, is not an ideal (notice that $0=\min(I\setminus\{x\})$), then there exists $y\in I\setminus\{x\}$ and $s\in S$, such that $y+s\not\in I\setminus\{x\}$. But $y\in I$, and thus $y+s\in I$, which forces $y\neq y+s=x$, contradicting that $x$ is minimal in $I$ with respect to $\le_S$.
\end{proof}

Given two ideals $I$ and $J$ in $\mathfrak{I}_0(S)$ we say that $I$ \emph{covers} $J$ if $J\subsetneq I$ and there is no other $K\in \mathfrak{I}_0(S)$ such that $J\subsetneq K\subsetneq I$.

\begin{lemma}\label{lem:covers-dif-one}
    Let $S$ be a numerical semigroup and let $I,J\in \mathfrak{I}_0(S)$, with $J\subsetneq I$. Then, $I$ covers $J$ if and only if $|I\setminus J|=1$.
\end{lemma}
\begin{proof}
    Suppose that there is $K\in \mathfrak{I}_0(S)$ such that $J\subsetneq K\subsetneq I$. Take $x\in I\setminus K$ and $y\in K\setminus J$. Then, $x\neq y$ and $x,y \in I\setminus J$, which forces $|I\setminus J|\ge 2$.

    For the converse, suppose that $|I\setminus J|\ge 2$. Let $m=\max(I\setminus J)$. Then, $J\subsetneq J\cup\{m\}\subsetneq I$. Let $x\in J\cup\{m\}$ and $s\in S$. If $x\in J$, then $x+s\in J\subset J\cup\{m\}$; if $x=m$ and $s=0$, then $x+s=m\in J\cup\{m\}$; if $x=m$ and $s\in S^*$, then $m<x+s\in I$, which implies that $x+s\in J\subset J\cup\{m\}$. Thus, $J\cup\{m\}\in\mathfrak{I}_0(S)$ and consequently $I$ does not cover $J$.
\end{proof}


\begin{lemma}\label{lem:number-covered}
    Let $S$ be a numerical semigroup and let $I\in \mathfrak{I}_0(S)$. Then, the number of ideals in $\mathfrak{I}_0(S)$ covered by $I$ equals the number of non-zero minimal generators of $I$.
\end{lemma}
\begin{proof}
    The proof easily follows from Lemmas~\ref{lem:remove-min-gen} and \ref{lem:covers-dif-one}.
\end{proof}

\begin{theorem}
    Let $S$ and $T$ be numerical semigroups. If $(\mathfrak{I}_0(S),\subseteq)$ and $(\mathfrak{I}_0(T),\subseteq)$ are isomorphic, then $S=T$.
\end{theorem}
\begin{proof}
    Notice that by Lemma~\ref{lem:number-covered}, $\mathfrak{P}_0(S)$ is precisely the set of ideals in $\mathfrak{I}_0(S)$ that cover exactly one ideal in $\mathfrak{I}_0(S)$. Hence, the isomorphism between $(\mathfrak{I}_0(S),\subseteq)$ and $(\mathfrak{I}_0(T),\subseteq)$ restricted to $(\mathfrak{P}_0(S),\subseteq)$ and $(\mathfrak{P}_0(T),\subseteq)$, yields, by Lemma~\ref{lem:prec-inc-principal}, an isomorphism between $(\operatorname{G}(S),\le_S)$ and $(\operatorname{G}(T),\le_T)$. By Corollary~\ref{cor:gap-poset-isomorphic}, we conclude that $S=T$.
\end{proof}

\section{Isomorphic ideal class monoids}

In this section, we prove that if $S$ and $T$ are numerical semigroups, then the existence of an isomorphism between $(\mathfrak{I}_0(S),+)$ and $(\mathfrak{I}_0(T),+)$ forces $S$ and $T$ to be equal. 

We start by proving that some notable elements of the ideal class monoid of a numerical semigroup are preserved under isomorphisms. To this end, we recall some definitions given in \cite[Section~5]{c-da-gs}; for the definitions on a general monoid, please refer to \cite{t}. 

Given $I,J\in \mathfrak{I}_0(S)$, we write $I\preceq J$ if there exists $K\in \mathfrak{I}_0(S)$ such that $I+K=J$. We use the notation $I\prec J$ when $I\preceq J$ and $I\neq J$ (in general this is not the usual definition, though in \cite[Section~5]{c-da-gs} it is shown that in our setting the usual definition is equivalent to this one). 

We say that $I\in \mathfrak{I}_0(S)$ is \emph{irreducible} if $I\neq J+K$ for all $J,K\in\mathfrak{I}_0(S)\setminus\{S\}$ such that $J\prec I$ and $K\prec I$. By \cite[Lemma~5.4]{c-da-gs}, $I$ is irreducible if and only if $I\neq J+K$ for all $J,K\in\mathfrak{I}_0(S)\setminus \{I\}$. Irreducible elements are important since they generate $(\mathfrak{I}_0(S),+)$ as a monoid \cite[Proposition~5.5]{c-da-gs}. Clearly, if $f:\mathfrak{I}_0(S)\to \mathfrak{I}_0(T)$ is a monoid isomorphism, then it sends irreducible elements to irreducible elements.

An ideal $I\in \mathfrak{I}_0(S)$ is a \emph{quark} if there is no ideal $J\in\mathfrak{I}_0(S)\setminus\{S\}$ such that $J\prec I$, that is, there is no $J\in\mathfrak{I}_0(S)\setminus\{I,S\}$ and $K\in \mathfrak{I}_0(S)$ such that $I=J+K$. Every quark is irreducible but the converse does not hold in general (see for instance \cite[Example~5.3]{c-da-gs}). Again, quarks go to quarks under monoid isomorphisms of ideal class monoids. 

The concepts of over-semigroup and irreducible numerical semigroup are crucial in the proof of the main result of this section. So, next we spend some time recalling the basic facts associated to these notions.

Let $S$ and $T$ be numerical semigroups. We say that $T$ is an \emph{over-semigroup} of $S$ if $S\subseteq T$. By \cite[Proposition~5.14]{c-da-gs},  $T$ is an idempotent of $\mathfrak{I}_0(S)$ if and only if $T$ is an over-semigroup of $S$. Denote by $\mathcal{O}(S)$ the set of over-semigroups of $S$. Then, $(\mathcal{O}(S),+)$ is a submonoid of $(\mathfrak{I}_0(S),+)$.

If $T$ is an over-semigroup of $S$ with $|T\setminus S|=1$, then we say that $T$ is a \emph{unitary extension} of $S$. In this setting, $T=S\cup\{x\}$, and $x$ must be a \emph{special gap} of $S$, that is $x\in\operatorname{FP}(S)$ and $2x,3x\in S$. The set of special gaps of $S$ is denoted by $\operatorname{SG}(S)$ and its cardinality coincides with the set of unitary extensions of $S$ (see for instance \cite[Section~3.3]{ns}). 

A numerical semigroup $S$ is \emph{irreducible} if $S$ cannot be expressed as the intersection of two numerical semigroups properly containing $S$. Every irreducible numerical semigroup is either symmetric or pseudo-symmetric. A numerical semigroup $S$ is \emph{symmetric} if for every $z\in \mathbb{Z}\setminus S$, the integer $\operatorname{F}(S)-z$ is in $S$. And it is \emph{pseudo-symmetric} if $\operatorname{F}(S)$ is even and for every $z\in \mathbb{Z}\setminus(S\cup\{\operatorname{F}(S)/2\})$, we have that $\operatorname{F}(S)-z\in S$. If $S$ is not irreducible, then it can be expressed as the intersection of finitely many irreducible over-semigroups of $S$ (for basic characterizations of irreducibility, symmetry and pseudo-symmetry, see \cite[Chapter~2]{ns-ap} or \cite[Chapter~3]{ns}). 

It is well known that $S$ is irreducible if and only if the cardinality of $\operatorname{SG}(S)$ is at most one \cite[Corollary~4.38]{ns}. If $S\neq \mathbb{N}$, then $\operatorname{F}(S)\in \operatorname{SG}(S)$. Thus, for $S\neq \mathbb{N}$, if $S$ is irreducible, then $S\cup\{\operatorname{F}(S)\}$ is the only unitary-extension of $S$, and thus every proper over-semigroup of $S$ contains $S\cup\{\operatorname{F}(S)\}$.

Quarks are relevant since they can be used to decide if the semigroup is symmetric or pseudo-symmetric, and thus to determine if the semigroup is irreducible; see Propositions~5.17 and 5.19, and Theorem~5.20 in \cite{c-da-gs}. Unitary extensions of a numerical semigroup $S$ are precisely the idempotent quarks of $\mathfrak{I}_0(S)$ \cite[Propostion~5.13]{c-da-gs}.

For every idempotent $E\in \mathfrak{I}_0(S)$, define 
\[
C_E=\{ I\in \mathfrak{I}_0(S) : I+E=I\}.
\]
This definition is inspired by \cite[Section~2]{g-z}.

\begin{proposition}\label{prop:CT-IT}
    Let $S$ be a numerical semigroup, and let $T$ be an over-semigroup of $S$. Then, \[C_T=\mathfrak{I}_0(T).\]
\end{proposition}
\begin{proof}
    Let $I\in C_T$. Then, $\min(I)=0$ and $I+T=I$, whence $I\in \mathfrak{I}_0(T)$.
    Now, let $I\in \mathfrak{I}_0(T)$. Then, $\min(I)=0$ and $I+T\subseteq I$. Hence, $I+S\subseteq I+T\subseteq I$, and thus $I\in \mathfrak{I}_0(S)$. As $I\subseteq I+T\subseteq I$, we get $I+T=I$, which yields $I\in C_T$.
\end{proof}

Let $S$ be a numerical semigroup. Observe that if $f$ is an isomorphism between $(\mathfrak{I}_0(S),+)$ and $(\mathfrak{I}_0(T),+)$, then from Proposition~\ref{prop:CT-IT} (and its proof), we obtain the following consequences.
\begin{enumerate}[(1)]
    \item The restriction of $f$ to $\mathcal{O}(S)$ is an isomorphism between $(\mathcal{O}(S),+)$ and $\mathcal{O}(T),+)$.
    \item  Also, for $O$ and $O'$ over-semigroups of $S$, $O\subseteq O'$ if and only if $O+O'=O'$. Thus, we also obtain an isomorphism between the posets $(\mathcal{O}(S),\subseteq)$ and $(\mathcal{O}(T),\subseteq)$.
    \item If $O$ is an over-semigroup of $S$, then $f(C_O)=C_{f(O)}$. To prove this, take $I\in f(C_O)$. Then, there exists $J\in C_O$ such that $I=f(J)$. As $J+O=J$, we deduce that $I+f(O)=f(J)+f(O)=f(J+O)=f(J)=I$, and thus $I\in C_{f(O)}$. For the other inclusion, let $J\in C_{f(O)}$. Then, as $f$ is surjective, there exists $I\in \mathfrak{I}_0(S)$ such that $f(I)=J$. Since $J+f(O)=J$, we have $f(I+O)=f(I)+f(O)=J+f(O)=J=f(I)$, and as $f$ is injective, $I+O=I$, which means that $I\in C_O$ and so $J=f(I)\in f(C_O)$. Therefore, the restriction of $f$ to $\mathfrak{I}_0(O)$ is an isomorphism between $(\mathfrak{I}_0(O),+)$ and $(\mathfrak{I}_0(f(O)),+)$.    
\end{enumerate}

Unfortunately, from the poset $(\mathcal{O}(S),\subseteq)$ it is not possible to recover $S$ as the next example shows. As usual, for a set $A$ of non-negative integers, we denote by 
\[\langle A\rangle = \{ n_1a_1+\dots+n_ka_k : k\in \mathbb{N}, n_1,\dots,n_k\in \mathbb{N}, a_1,\dots,a_k\in A\},\]
which is a submonoid of $(\mathbb{N},+)$, and it is a numerical semigroup if and only if $\gcd(A)=1$ (see for instance \cite[Lemma~2.1]{ns}).

\begin{example}
 The numerical semigroups $\langle 3,5,7\rangle$ and $\langle 2,7\rangle$ have isomorphic posets of over-semi\-groups with respect to set inclusion.
\end{example}

Notice that if $E$ is an idempotent of $\mathfrak{I}_0(S)$, then $(C_E,+)$ is a monoid, but it is not a submonoid of $(\mathfrak{I}_0(S),+)$ unless $E=S$. There is a dual construction that allows us to construct submonoids of $(\mathfrak{I}_0(S),+)$ associated to its idempotents. Let $T$ be an over-semigroup of $S$. Then, $T\in \mathfrak{I}_0(S)$ and $T$ is idempotent. Define 
\[
T_\downarrow = \{ I \in \mathfrak{I}_0(S) : I \subseteq T\}.
\]

\begin{proposition}
    Let $S$ be a numerical semigroup, and let $T$ be an over-semigroup of $S$. For every $I\in \mathfrak{I}_0(S)$, $I\in T_\downarrow$ if and only if $I+T=T$. In particular, $(T_\downarrow,+)$ is a submonoid of $(\mathfrak{I}_0(S),+)$.
\end{proposition}
\begin{proof}
    Let $I\in \mathfrak{I}_0(S)$. If $I+T=T$, as $I\subseteq I+T$, we obtain that $I\subseteq T$. Now, let $I\in \mathfrak{I}_0(S)$ with $I\subseteq T$. Then, $T\subseteq I+T\subseteq T+T=T$ (recall that $T$ is idempotent), and so $I+T=T$.

    Finally, take $I,J\in T_\downarrow$. Then $(I+J)+T=I+(J+T)=I+T=T$, and so $I+J\in T_\downarrow$. The identity element of $(T_\downarrow,+)$ is $S$.
\end{proof}

Take $T$ and $T'$ two over-semigroups of $S$ (equivalently, two idempotents of $\mathfrak{I}_0(S)$) with $T\subseteq T'$ (equivalently, $T+T'=T'$). Then, $(T'_\downarrow\cap C_T,+)$ is a monoid with identity element $T$.

\begin{example}
    Let $I_1=\{0,7\}+S$ and $I_2=\{0,2,3,5\}+S$. Both $I_1$ and $I_2$ are idempotents, and thus they are over-semigroups of $S$; moreover, $I_1\subseteq I_2$. The monoid $({I_2}_\downarrow \cap C_{I_1},+)$ has eight elements, and its Hasse diagram with respect to $\preceq$ has height four (the maximal strictly ascending chain with respect to $\preceq$ has length four). According to \cite[Remark~5,1]{c-da-gs}, if $({I_2}_\downarrow \cap C_{I_1},+)$ is isomorphic to $\mathfrak{I}_0(T)$ for some semigroup $T$, then the genus of $T$ should be three. Among the semigroups of genus three, the only one whose ideal class monoid has cardinality eight is $T=\langle 4,5,6,7\rangle$. However, $\mathfrak{I}_0(T)$ has only three irreducible elements while $({I_2}_\downarrow \cap C_{I_1},+)$ has four irreducible elements. Thus, $({I_2}_\downarrow \cap C_{I_1},+)$ is not isomorphic to the ideal class monoid of a numerical semigroup. 
\end{example}

We are now ready to prove that if $(\mathfrak{I}_0(S),+)$ is isomorphic to $(\mathfrak{I}_0(T),+)$, with $S$ and $T$ numerical semigroups, then $S=T$. To this end, we proceed by induction on the genus of $S$ (which must be the same as the genus of $T$ by \cite[Corollary~5.2]{c-da-gs}). Once we know $(\mathfrak{I}_0(S),+)$ all the unitary extensions of $S$ will be uniquely determined by the induction hypothesis.



\begin{lemma}\label{lem:intersect-unitary-extensions}
    Let $S$ be a numerical semigroup, $S\neq \mathbb{N}$. If $S$ is irreducible, then the intersection of all its unitary extensions is $S\cup\{\operatorname{F}(S)\}$. Otherwise, this intersection is $S$.
\end{lemma}
\begin{proof}
    Recall that as $S\neq \mathbb{N}$, we have that $\operatorname{F}(S)\in \operatorname{SG}(S)$. Notice that $S$ is irreducible if and only if $|\operatorname{SG}(S)|=1$ \cite[Corollary~4.38]{ns}. Also, every unitary extension of $S$ is of the form $S\cup\{h\}$ with $h\in \operatorname{SG}(S)$. Hence, if $S$ is irreducible, then the intersection of all its unitary extensions (it has only one), is $S\cup\{\operatorname{F}(S)\}$. If $S$ is not irreducible, then take $h\in \operatorname{SG}(S)\setminus\{\operatorname{F}(S)\}$. Clearly, $S=(S\cup\{\operatorname{F}(S)\})\cap(S\cup\{h\})$, and thus $S=\bigcap_{g\in \operatorname{SG}(S)} (S\cup\{g\})$.
\end{proof}

\begin{theorem}\label{thm:isomorphic-class-monoids-addition}
    Let $S$ and $T$ be numerical semigroups. If $(\mathfrak{I}_0(S),+)$ is isomorphic to $(\mathfrak{I}_0(T),+)$, then $S=T$.
\end{theorem}
\begin{proof}
    Denote by $\varphi$ the isomorphism between $\mathfrak{I}_0(S)$ and $\mathfrak{I}_0(T)$.
    
    Notice that $S=\mathbb{N}$ if and only if $\mathfrak{I}_0(S)$ is trivial. Thus, we may assume that $S$ and $T$ are different from $\mathbb{N}$. Notice also that the number of quarks of $\mathfrak{I}_0(S)$ and $\mathfrak{I}_0(T)$ must be the same. Thus, $S$ is irreducible if and only if $T$ is irreducible. Also, by \cite[Corollary~5.2]{c-da-gs}, $\operatorname{g}(S)=\operatorname{g}(T)$. 

    We proceed by induction on the genus of $S$ (which is the same as the genus of $T$). For $\operatorname{g}(S)=0$ there is nothing to prove, since in this case $S=T=\mathbb{N}$. So, suppose that the assertion is true for all semigroups having genus $g$ and let us prove it for genus $g+1$. 

    Unitary extensions of $S$ correspond to idempotent quarks in $\mathfrak{I}_0(S)$ \cite[Proposition~5.13]{c-da-gs}. Thus, for every unitary extension $O$ of $S$, $\varphi(O)$ is also unitary extension of $T$, and by Proposition~\ref{prop:CT-IT}, $\mathfrak{I}_0(O)=C_O$ is isomorphic to $C_{\varphi(O)}=\mathfrak{I}_0(\varphi(O))$. Unitary extensions of $S$ have genus $g$ (the same holds for $T$). Thus, by induction hypothesis $S$ and $T$ have the same unitary extensions.

    If $S$ is not irreducible, then $T$ cannot be irreducible by the arguments given above. As $S$ and $T$ are not irreducible, they are the intersection of all their unitary extensions, and consequently $S=T$ (Lemma~\ref{lem:intersect-unitary-extensions}).

    If $S$ is irreducible and symmetric, then so is $T$, since in this setting both have a single quark \cite[Proposition~5.18]{c-da-gs}. In this case, by Lemma~\ref{lem:intersect-unitary-extensions}, the intersection of the unitary extensions of $S$ is $S\cup\{\operatorname{F}(S)\}$, which must be equal to $T\cup\{\operatorname{F}(T)\}$. We also know that $\operatorname{g}(S)=\operatorname{g}(T)$, and as $S$ and $T$ are symmetric, by \cite[Corollary~6]{ns-ap}, $\operatorname{F}(S)=2\operatorname{g}(S)-1=2\operatorname{g}(T)-1=\operatorname{F}(T)$. Thus, $S=(S\cup\{\operatorname{F}(S)\})\setminus\{\operatorname{F}(S)\}=(T\cup\{\operatorname{F}(T)\})\setminus\{\operatorname{F}(T)\}=T$.

    The remaining case is when $S$ and $T$ are both irreducible and pseudo-sy\-mme\-tric ($S$ is pseudo-symmetric if and only if $\mathfrak{I}_0(S)$ has two quarks; see \cite[Proposition~5.19]{c-da-gs}). In this setting, by using again \cite[Corollary~6]{ns-ap}, $\operatorname{F}(S)=2\operatorname{g}(S)-2=2\operatorname{g}(T)-2=\operatorname{F}(T)$, and arguing as in the preceding paragraph, we conclude that $S=T$.
\end{proof}

\section{The poset of the ideal class monoid induced by addition}

We solved \cite[Question~6.1]{c-da-gs}, but we still do not know how to recover a numerical semigroup by looking at a poset isomorphic to $(\mathfrak{I}_0(S),\preceq)$ \cite[Question~6.2]{c-da-gs}. There are several options to tackle this problem. The first could be to recover $\subseteq$ from $\preceq$, while the second could pass through identifying idempotent quarks in the Hasse diagram of $(\mathfrak{I}_0(S),\preceq)$. 

Clearly, if $I\preceq J$, then $I\subseteq J$. But the converse does not hold. Actually, $J$ covers $I$ with respect to set inclusion if and only if $|J\setminus I|=1$ (Lemma~\ref{lem:covers-dif-one}). However, it may happen that $J$ covers $I$ with respect to $\preceq$ and $|J\setminus I|>1$. 

\begin{example}
Take $S=\langle 5,9,17,21\rangle$, $I=\{0,3\}+S$ and $J=\{0,12\}+S$. Then, $I$ $\preceq$-covers $J$ and $|I\setminus J|=3$. This example was obtained by looking at the Hasse diagram of $(\mathfrak{I}_0(S),\preceq)$.

\end{example}

Figure~\ref{fig:hasse-preceq-4-6-9} shows the Hasse diagram of $(\mathfrak{I}_0(\langle 4,6,9\rangle),\preceq)$. The edges displayed as a dashed line are not part of the diagram, and correspond to the coverings with respect to set inclusion that are not coverings with respect to $\preceq$. A possible approach would be ``repair'' those missing edges. 

\begin{figure}
\centering
\includegraphics[scale=0.3]{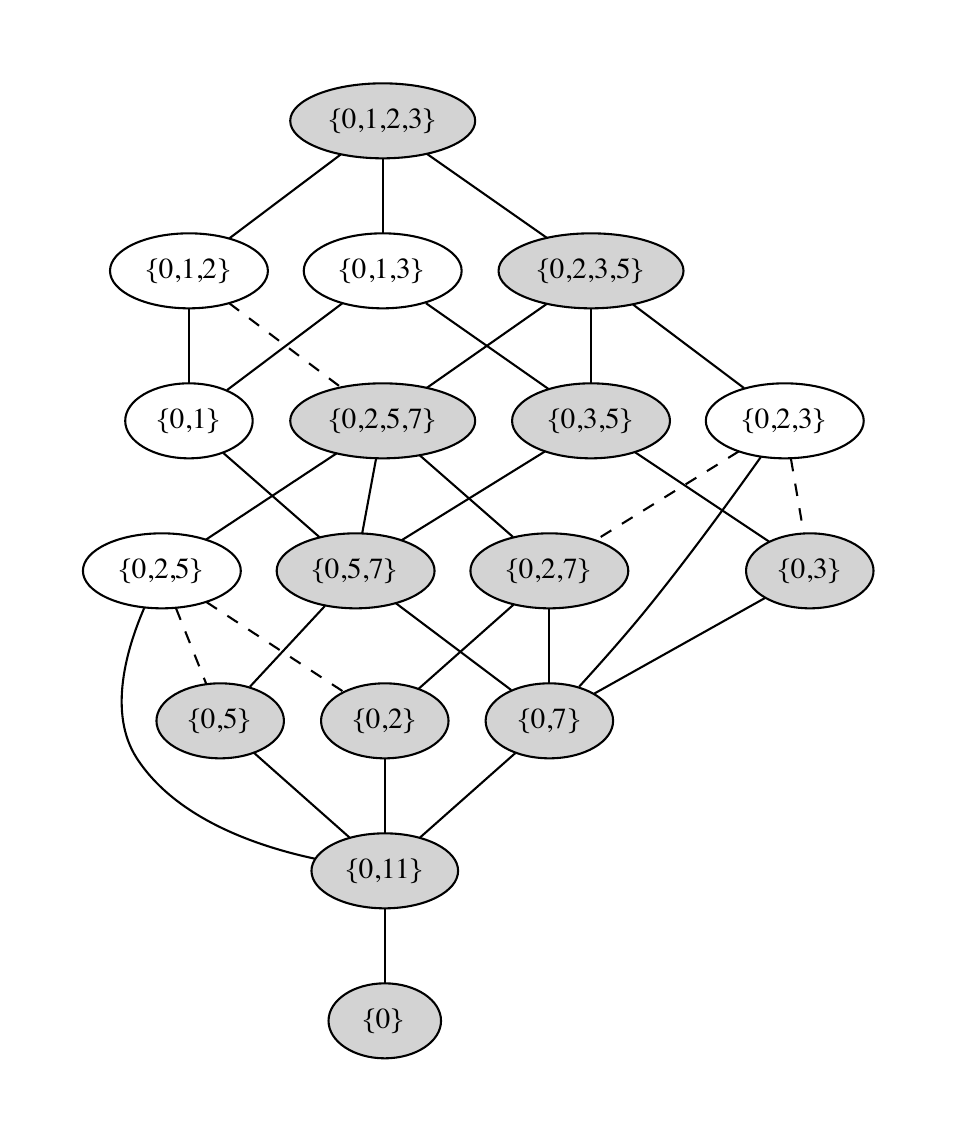}

\caption{The Hasse diagram of $(\mathfrak{I}_0(\langle 4,6,9\rangle,\preceq)$; the nodes are labeled with the minimal generating sets of the ideals. Dashed edges are those edges missing from the Hasse diagram of $(\mathfrak{I}_0(\langle 4,6,9\rangle),\subseteq)$. Idempotents are displayed in gray.}
\label{fig:hasse-preceq-4-6-9}
\end{figure}

As for the other approach. Quarks are easy to distinguish in the poset $(\mathfrak{I}_0,\preceq)$, since they are the ones with ``height'' one. However, even in simple examples, it is not possible to discern from the poset which ones are idempotent. The Hasse diagram of $\mathfrak{I}_0(\langle 3,4,5\rangle)$ is shown in Figure~\ref{fig:hasse-preceq-3-4-5}. It is not possible from the poset with respect to $\preceq$ to distinguish between $\{0,1\}+S$ and $\{0,2\}+S$; the latter being idempotent, while the first is not. Notice that in this case the genus is two, and there are only two numerical semigroups with this genus. The posets of the corresponding set of normalized ideals are different.

\begin{figure}
\includegraphics[scale=0.3]{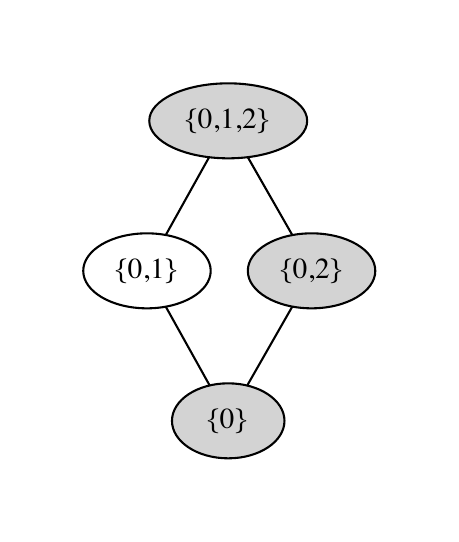}
\caption{The Hasse diagram of $(\mathfrak{I}_0(\langle 3,4,5\rangle,\preceq)$; nodes labeled with the minimal generating sets of the ideals.  Idempotents are displayed in gray.}
\label{fig:hasse-preceq-3-4-5}
\end{figure}

\FloatBarrier

\section*{Acknowledgments}
The author is partially supported by the grant number ProyExcel\_00868 (Proyecto de Excelencia de la Junta de Andalucía) and by the Junta de Andaluc\'ia Grant Number FQM--343. He also acknowledges financial support from the grant PID2022-138906NB-C21 funded by MCIN/AEI/10.13039/\bignumber{501100011033} and by ERDF ``A way of making Europe'', and from the Spanish Ministry of Science and Innovation (MICINN), through the ``Severo Ochoa and María de Maeztu Programme for Centres and Unities of Excellence'' (CEX2020-001105-M).

The author would like to thank Marco D'Anna and Salvo Tringali for their comments and suggestions.

\end{document}